\documentclass[12pt,reqno,twoside]{amsart}
\usepackage{amsmath}
\usepackage{amssymb}
\usepackage{mathrsfs}
\usepackage{mathabx} %For the command $\bigtimes$
\usepackage{graphicx,epsfig,wrapfig,sidecap,subfig}
\usepackage[all]{xy}
\usepackage{proof}
\usepackage{xcolor}
\usepackage{hyperref}
\usepackage{bm}
\usepackage[capitalize]{cleveref}
\usepackage{listings}

\usepackage[a4paper, hmargin=2.5cm, vmargin={3.2cm, 3.2cm}]{geometry}

\usepackage{color}

\title[A dichotomy for BD equivalence of Delone sets]{A dichotomy for bounded displacement equivalence of Delone sets}
\date{}
\author{Yotam Smilansky}
\address{Yotam Smilansky\newline\indent Department of Mathematics, Rutgers University, NJ, USA. \newline\indent {\tt yotam.smilansky@rutgers.edu}}
\author{Yaar Solomon}
\address{Yaar Solomon\newline\indent Department of Mathematics, Ben-Gurion University of the Negev, Israel.\newline\indent {\tt yaars@bgu.ac.il}}

\newcommand{\N}{{\mathbb{N}}}
\newcommand{\Z}{{\mathbb{Z}}}

\newcommand{\R}{{\mathbb{R}}}

\newcommand{\X}{{\mathbb{X}}}

\newcommand{\EE}{{\mathcal{E}}}

\newcommand{\GG}{{\mathcal{G}}}

\newcommand{\QQ}{\mathcal{Q}}

\renewcommand{\aa}{\mathbf{a}}

\newcommand{\pp}{\mathbf{p}}
\newcommand{\qq}{\mathbf{q}}
\newcommand{\vv}{\mathbf{v}}
\newcommand{\uu}{\mathbf{u}}

\newcommand{\xx}{\mathbf{x}}
\newcommand{\yy}{\mathbf{y}}
\newcommand{\zz}{\mathbf{z}}

\newcommand{\CCC}{\mathscr{C}}

\newcommand{\cont}{2^{\aleph_0}}
\newcommand{\bdX}{ \mathrm{BD}(\X) }

\newcommand{\df}{{\, \stackrel{\mathrm{def}}{=}\, }}

\newcommand{\dist}{\mathbf{dist}}

\newcommand{\absolute}[1] {\left|{#1}\right|}

\newcommand{\norm}[1]{\left\|{#1}\right\|}

\newcommand{\vol}{\mathrm{vol}}

\newcommand{\ignore}[1]  {}

\theoremstyle{plain}
\newtheorem{thm}{Theorem}[section]

\newtheorem{lem}[thm]{Lemma}
\newtheorem{prop}[thm]{Proposition}
\newtheorem{cor}[thm]{Corollary}

\theoremstyle{definition}

\newtheorem*{question}{Question}

\numberwithin{equation}{section}
\swapnumbers

\newif\ifdraft\drafttrue
\usepackage{color}

\begin{document}
\begin{abstract}
      We prove that in every compact space of Delone sets in $\R^d$ which is minimal with respect to the action by translations, either all Delone sets are uniformly spread, or continuously many distinct bounded displacement equivalence classes are represented, none of which contains a lattice. The implied limits are taken with respect to the Chabauty--Fell topology, which is the natural topology on the space of closed subsets of $\R^d$. This topology coincides with the standard local topology in the finite local complexity setting, and it follows that the dichotomy holds for all minimal spaces of Delone sets associated with well-studied constructions such as cut-and-project sets and substitution tilings, whether or not finite local complexity is assumed. 
\end{abstract}

\maketitle

\section{Introduction}\label{sec:introduction}
A set $\Lambda \subset \R^d$ is called a \emph{Delone set} if it is both \emph{uniformly discrete} and \emph{relatively dense}, that is, if there are constants $r,R>0$ so that every ball of radius $r$ contains at most one point of $\Lambda$ and $\Lambda$ intersects every ball of radius $R$. We refer to $r$ and $R$ as the \emph{separation constant} and the \emph{packing radius} of $\Lambda$, respectively. 
Two Delone sets $\Lambda, \Gamma \subset \R^d$ are said to be \emph{bounded displacement (BD) equivalent}
%, denoted by $\Lambda \bd \Gamma$, 
if there exists a bijection $\phi: \Lambda \to \Gamma$ satisfying 
\[\sup_{\xx\in \Lambda}\norm{\xx-\phi(\xx)}<\infty.\] 
Such a mapping $\phi$ is called a \emph{BD-map}. Note that since norms in $\R^d$ are equivalent, this definition does not depend on the choice of norm. 
Lattices in $\R^d$ with the same covolume are BD-equivalent, and  a Delone set $\Lambda$ is called \emph{uniformly spread} if it is equivalent to a lattice, or equivalently, if there is a BD-map $\phi: \Lambda \to \alpha\Z^d$, for some $\alpha>0$.

Fix a metric $\rho$ on $\R^d$ and consider the space $\CCC(\R^d)$ of closed subsets of $\R^d$. The \textit{Chabauty--Fell topology} on $\CCC(\R^d)$ is the topology induced by the metric (see Appendix \ref{Appendix}) 
\begin{equation}\label{eq:CF-metric}
D\left(\Lambda_0, \Lambda_1\right) \df 
\inf\left(\left\{ \varepsilon>0 \: \Big| \: \begin{matrix}\Lambda_0 \cap B({\bf0},1/\varepsilon)\subset \Lambda_1^{(+\varepsilon)}\\
\Lambda_1 \cap B({\bf0},1/\varepsilon)\subset \Lambda_0^{(+\varepsilon)}
\end{matrix}\right\} \cup\{1\}\right),
\end{equation}
where $B(\xx,R)$ is the open ball of radius $R>0$ centered at $\xx\in\R^d$ with respect to the metric $\rho$, and $A^{(+\varepsilon)}$ is the $\varepsilon$ neighborhood of the set $A$. In this work we only consider metrics $\rho$ that are determined by norms on $\R^d$, and although different choices of norms result in different metrics $D$, they all define the same  %Chabauty--Fell 
topology. We remark that in the aperiodic order literature, this topology, which was introduced by Chabauty \cite{Ch} for $\CCC(\R^d)$ as well as for a more general setting, and later extended by Fell \cite{Fe}, is often referred to as the \textit{natural topology} or the \textit{local rubber topology}, see e.g. \cite[\S 5]{BaakeGrimm}. See also \cite{Lenz-Stollmann} for the relation to the Hausdorff metric.

Delone sets in $\R^d$ are elements of $\CCC(\R^d)$, and we may consider compact spaces of Delone sets, where the implied limits are taken with respect to the Chabauty--Fell topology. Such a space $\X$ of Delone sets in $\R^d$ is \textit{minimal} with respect to the $\R^d$ action by translations if the orbit closure of every Delone set $\Lambda\in\X$ is dense in $\X$.  Minimality of $\X$ is equivalent to the recurrence of {\it patches} in each Delone set $\Lambda\in\X$, where a patch is a finite subset of a Delone set. This important geometric consequence of minimality is called {\it almost repetitivity}, and a precise definition and additional details will be given in \S \ref{sec:topology}. For a proof of this equivalence see \cite[Theorem 3.11]{Frettloh-Richard} and \cite[Theorem 6.5]{Smilansky-Solomon}, and see also the discussion included in \cite{Kellendonk-Lenz}.

Denote the cardinality of the set of BD-equivalence classes represented in $\X$ by $\bdX$. The following dichotomy is our main result.

\begin{thm}\label{thm:main_result_general}
	Let $\X$ be a space of Delone sets in $\R^d$, and assume it is compact with respect to the Chabauty--Fell topology and minimal with respect to the action of $\R^d$ by translations. Then either 
	\begin{enumerate}
		\item 
		there exists a uniformly spread Delone set in $\X$ (and so every $\Lambda \in \X$ is uniformly spread and $\bdX = 1$), \\
		\hspace{2in} or
		\item
		$\bdX = \cont$,
	\end{enumerate} 
 	where $\cont$ denotes the cardinality of the continuum.
\end{thm}

Observe that the minimality assumption is essential, as shown by the following simple example. Consider $\Lambda = (-2\N) \sqcup \{0\} \sqcup \N$, a Delone set in $\R$. Then the orbit closure $\X$ of $\Lambda$ under translations by $\R$ and with respect to the Chabauty--Fell topology, consists of translations of $\Lambda$, the orbit closure of $\Z$ and the orbit closure of $2\Z$. Therefore $\bdX = 3$, and indeed $\X$ is not minimal.

Let us describe the proof of Theorem \ref{thm:main_result_general}. The implication in the brackets of (1) is a direct consequence of \cite[Theorem 1.1]{Laczk}, see also \cite[Theorem 3.2]{FG} for a sketch of a similar proof that holds for general minimal spaces of Delone sets. A uniformly discrete set in $\R^d$ with separation constant $r>0$ is BD-equivalent to a subset of the lattice $\frac{r}{2}\Z^d$, hence the upper bound $\bdX \le \cont $ is trivial.  We prove the remaining implication according to the following steps. Given a non-uniformly spread Delone set in a minimal space $\X$, we construct in \S \ref{sec:construction_of_the_patches} a sequence of pairs of patches consisting of an increasingly deviant number of points. by \S \ref{sec:topology}, choosing a patch from each pair, which corresponds to the choice of a word on a two letter alphabet, gives rise to a Delone sets in $\X$  with certain properties. Finally, it is shown in \S \ref{sec:proof of main} that using the equivalent condition for non-BD equivalence of two Delone sets established in \S \ref{sec:conditions_for_BD-non-equivalence}, two Delone sets defined using words that differ in infinitely many places are BD-non-equivalent, and so $\X$ contains continuously many BD-equivalence classes. y many places must be BD-non-equivalent.

Recall that a Delone set $\Lambda$ has \emph{finite local complexity (FLC)} if for every $R>0$ the number of distinct patterns that are contained in balls of radius $R$ in $\Lambda$ up to translations is finite. In such case every Delone set in the orbit closure of $\Lambda$ under translations, sometimes called the \emph{hull} of $\Lambda$, also has FLC. The hull itself is then called FLC, and the Chabauty--Fell topology on $\X$ coincides with the \emph{local topology}, see \cite[\S 5]{BaakeGrimm}. It follows that Theorem \ref{thm:main_result_general} holds also for FLC spaces with respect to the local topology, and constitutes a new result both in the FLC and non-FLC setup. In particular, it answers question (1) in \cite[\S 7]{FG} in the strongest possible way. 

In addition to Theorem \ref{thm:main_result_general}, we establish in Theorem \ref{thm:non-BD_consequence} a useful equivalent condition for two Delone sets to be non-BD equivalent. This result is the converse of the implication of Theorem \ref{thm:non-BD_condition} which first appeared in \cite{FSS}, and may be of interest in its own right.

Delone sets are mathematical models of atomic positions, and BD-equivalence offers a natural way of classifying them. BD-equivalence for general discrete point sets was previously considered mainly in the context of uniformly spread point sets, see e.g. \cite{DO1, DO2, Laczk} and \cite{DSS}. 
In recent years, BD-equivalence emerged as an object of study for Delone sets that appear in the study of mathematical quasicrystals and aperiodic order, see \cite{BaakeGrimm} for a comprehensive introduction to such constructions. 
For cut-and-project sets, BD-equivalence was studied in \cite{HKW}, and links to the notions of \emph{bounded remainder sets} and \emph{pattern equivariant cohomology} appeared in \cite{FG, HK,HKK} and in \cite{KS1,KS2}, respectively. For Delone sets associated with substitution tilings, sufficient conditions for a set to be uniformly spread were provided in \cite{ACG}, \cite{Solomon11} and \cite{Solomon14}. In addition, for the multiscale substitution tilings introduced by the authors in \cite{Smilansky-Solomon}, it was shown that any Delone set associated with an incommensurable tiling cannot be uniformly spread.

Recently, questions regarding BD-non-equivalence between two Delone sets were considered in \cite{FSS}, where a sufficient condition for BD-non-equivalence was established. It was later shown in \cite{Solomon20}  that if the eigenvalues and eigenspaces of the substitution matrix satisfy a certain condition, then the corresponding substitution tiling space contains continuously many distinct BD-classes.

The following less restrictive equivalence relation on Delone sets is often studied in parallel to the BD-equivalence relation. We say that two Delone sets $\Lambda$ and $\Gamma$ are \emph{biLipschitz (BL) equivalent} if there exists a biLipschitz bijection between them. Namely, a bijection $\varphi:\Lambda \to \Gamma$ and a constant $C\ge 1$ so that 
\[\forall \xx,\yy\in \Lambda \qquad \frac1C \le \frac{\norm{\varphi(\xx) - \varphi(\yy)}}{\norm{\xx - \yy}} \le C. \]
It was shown by Burago and Kleiner \cite{BK1} and independently by McMullen  \cite{McMullen}, that there exist Delone sets in $\R^d$, $d\ge 2$, that are not BL-equivalent to a lattice in $\R^d$. It was shown in \cite{Magazinov} that there are continuously many Delone sets that are pairwise BL-non-equivalent, and a hierarchy of equivalence relations on Delone sets, which includes BD and BL equivalence, was recently introduced in \cite{DK}. It would be interesting to obtain an analogue of our Theorem \ref{thm:main_result_general} in this context.   
\begin{question}
	Does Theorem \ref{thm:main_result_general} hold if  BD-equivalence is replaced by BL-equivalence?
\end{question}
In view of the sufficient condition for BL-equivalence to a lattice given by Burago and Kleiner in \cite{BK2} and the constructions in \cite{Cortez-Navas} and \cite{Magazinov}, we remark that the results given in \S\ref{sec:topology} and \S\ref{sec:construction_of_the_patches} regarding densities and discrepancy estimates may be relevant also in the study of BL-non-equivalence and the question stated above.

\subsection{Consequences of Theorem \ref{thm:main_result_general}} 
Theorem \ref{thm:main_result_general} directly implies that $\bdX = \cont$ for many special families of minimal spaces of Delone sets which are central in the theory of aperiodic order, and for which the BD-equivalence relation was previously considered.  

\subsubsection{Substitution tilings:}
	For primitive substitution tilings of $\R^d$, we denote by $\lambda_1>\absolute{\lambda_2} \ge\ldots\ge  \absolute{\lambda_n}$ the eigenvalues of the substitution matrix, and we let $t\ge 2$ be the minimal index such that the eigenspace of $\lambda_t$ contains non-zero vectors whose sum of coordinates is not zero. 
	Under the assumption that tiles are bi-Lipschitz homeomorphic to closed balls, it was shown in \cite[Theorem 1.2 (I)]{Solomon14} that if
	\begin{equation}\label{eq:eigenvalues_of_substitution_for_non-BD}
	\absolute{\lambda_t} > \lambda_1^{(d-1)/d}
	\end{equation} 
	then the Delone sets corresponding to the tilings in the tiling space are not uniformly spread. Under the assumption \eqref{eq:eigenvalues_of_substitution_for_non-BD} and an additional assumption regarding the existence of certain patches, it was recently shown in \cite{Solomon20} that $\bdX = \cont$. Given the above result of \cite{Solomon14}, and since substitution tiling spaces are minimal (see \cite{BaakeGrimm}), the following strengthening of the main result of \cite{Solomon20} is a direct consequence of our Theorem \ref{thm:main_result_general}. 
	\begin{cor}
		Let $\X$ be a primitive substitution tiling space with tilings by tiles that are bi-Lipschitz homeomorphic to closed balls. Assume that condition \eqref{eq:eigenvalues_of_substitution_for_non-BD} holds,  then $\bdX = \cont$.
	\end{cor}
	Note that in the context of tilings, we say that two tilings are BD-equivalent if their corresponding Delone sets, which are obtained by picking a point from each tile, are BD-equivalent. 
	In addition to the above, \cite{Solomon14} contains an example of a substitution rule, for which the eigenvalues of the substitution matrix satisfy 
	\begin{equation}\label{eq:eigenvalues_of_substitution_with_equality}
	\absolute{\lambda_2} = \lambda_1^{(d-1)/d}
	\end{equation} 
	and the corresponding Delone sets are not uniformly spread, see \cite[Theorem 1.2 (III)]{Solomon14}. Note that in this example the main result of \cite{Solomon20} cannot be applied. \begin{cor}
		There exists a primitive substitution tiling space $\X$ for which condition  \eqref{eq:eigenvalues_of_substitution_with_equality} holds and $\bdX = \cont$.
	\end{cor}
	
\subsubsection{Cut-and-project sets:}
	Theorem 1.2 in \cite{HKW} concerns the BD-equivalence relation in the context of cut-and-project sets that arise from linear toral flows (which constitute an equivalent method of constructing cut-and-project sets, see \cite[Proposition 2.3]{ASW}). Since the hull of a cut-and-project set is minimal, the corollary below follows directly from \cite[Theorem 1.2 (III)]{HKW} and our Theorem \ref{thm:main_result_general}. We refer to \cite{HKW} for more details on the construction and terminology.
	\begin{cor}
		For almost every $(k-d)$-dimensional linear section $S$, which is a parallelotope in the $k$-dimensional torus, there is a
		residual set of $d$-dimensional subspaces V for which the hull of the corresponding cut-and-project set contains continuously many distinct BD-classes.
	\end{cor}

	 The \emph{half-Fibonacci sets} were introduced in \cite[\S 6]{FG}. These are cut-and-project sets in $\R$ that belong to the same hull and are BD-non-equivalent. In particular, they are not uniformly spread (see \cite[Theorem 3.2]{FG}). We thus obtain the following result. 
	\begin{cor}
		Let $\X$ be the hull of the half-Fibonacci sets from \cite{FG}. Then $\bdX = \cont$.
	\end{cor}
	
\subsubsection{Multiscale substitution tilings:}
	Multiscale substitution tilings were recently studied in \cite{Smilansky-Solomon}. Under an incommensurability assumption on the underlying substitution scheme the corresponding tiling spaces are minimal \cite[\S 6]{Smilansky-Solomon}, and combined with a mild assumption on the boundaries of the prototiles which holds for example for polygonal tiles, their associated Delone sets, which are never FLC, are also never uniformly spread \cite[\S 8]{Smilansky-Solomon}.  
	\begin{cor}
		Let $\X$ be an incommensurable multiscale polygonal tiling space. Then $\bdX = \cont$.
	\end{cor}  

A proof of Theorem \ref{thm:main_result_general} in the FLC setup was given in \cite{FGS}, which appeared after the first version of this paper came out. Their work is independent of ours. 

\subsection*{Acknowledgments} 
We are grateful to the anonymous referee for many insightful suggestions and remarks.

\section{Necessary and sufficient conditions for BD-non-equivalence}\label{sec:conditions_for_BD-non-equivalence}
\subsection{Notations.}
Bold figures will be used to denote vectors in $\R^d$, and we will use the supremum norm $\norm{\cdot}_\infty$  on $\R^d$ throughout this document. Note that with respect to this norm, balls are (Euclidean) cubes, and we use both terms interchangeably. We denote by $\partial A, \absolute{A}$ and $\vol(A)$ the boundary, cardinality and Lebesgue measure of a set $A\subset\R^d$, respectively, and we denote by $\# S$ the cardinality of a finite set $S$. Given $\varepsilon>0$ and $A\subset\R^d$ we denote the $\varepsilon$-neighborhood of $A$ by
\[A^{(+\varepsilon)}\df\{\xx \in \R^d \mid \dist(\xx,A)\le \varepsilon\}, \]
where $\dist(\xx,A) = \inf\{\norm{\xx - \aa}_\infty \mid \aa\in A\}$. 
For an integer $m>0$ we denote by  
\[
	\QQ_d(m) \df \left\{ \bigtimes_{i=1}^d [a_i,a_i+m) \mid a_1,\ldots,a_d\in m\Z \right\},
\]
the collection of all half-open cubes in $\R^d$ with edge-length $m$ and with vertices in $m\Z^d$, and we denote by $\QQ_d^*(m)$ the collection of finite unions of elements from $\QQ_d(m)$. In the case $m=1$ we simply write $\QQ_d$ and $\QQ_d^*$. For $A\in\QQ_d$ the notation $\vol_{d-1}(\partial A)$ stands for the $(d-1)$-Lebesgue measure of $\partial A$. The following lemma is a direct consequence of Lemmas 2.1 and 2.2 of \cite{Laczk}. 
\begin{lem}\label{lem:Lacz.Lemmas_2.1+2.2}
	Let $F$ be a translated copy of an element of $\QQ_d^*$ and let $s>0$, then 
	\begin{equation}\label{eq:Lacz.Lemmas_2.1+2.2}
	\vol\left( (\partial F)^{(+s)}\right) \le c_0 \cdot s^d\cdot \vol_{d-1}(\partial F),
	\end{equation} 
	where $c_0$ depends only on $d$. 
\end{lem}

\subsection{BD-equivalence}\label{subsec:BD}
The following condition for non-BD-equivalence of two Delone sets in $\R^d$ was given in \cite{FSS}.

\begin{thm}\cite[Theorem 1.1]{FSS}\label{thm:non-BD_condition}
	Let $\Lambda_0, \Lambda_1$ be two Delone sets in $\R^d$ and suppose that there is a sequence $(A_m)_{m\in\N}$ of sets, $A_m\in \QQ_d^*$, for which
	%\begin{equation}\label{eq:non_BD_condition}
	\[
		\frac{| \#(\Lambda_0 \cap A_m) - 
		\#(\Lambda_1 \cap A_m) |}{\vol_{d-1}(\partial A_m)} \xrightarrow{m\to\infty} \infty.
	\]
	%\end{equation}   
	Then there is no BD-map $\phi: \Lambda_0 \to \Lambda_1$.
\end{thm}

We show that the converse also holds (compare \cite[Lemma 2.3]{Laczk}).
\begin{thm}\label{thm:non-BD_consequence}
	The following are equivalent for two Delone sets $\Lambda_0, \Lambda_1$ in $\R^d$.
	\begin{itemize}
		\item[(i)]
		There is no BD-map between $\Lambda_0$ and $\Lambda_1$.
		\item[(ii)]
		There is a sequence  $(A_m)_{m\in\N}$ of sets, which are translated copies of elements of $\QQ_d^*$, such that %for every $m\in\N$ we have 
		\begin{equation}\label{eq:non_BD_consequence}
		\frac{| \#(\Lambda_0 \cap A_m) - \#(\Lambda_1 \cap A_m) |}{\vol_{d-1}(\partial A_m)}\xrightarrow{m\to\infty}\infty. 
		\end{equation}   		
	\end{itemize} 
\end{thm}

\begin{proof}
	The implication $(ii)\Rightarrow(i)$ follows from Theorem \ref{thm:non-BD_condition}, since translating the sets $A_m$ by at most $\sqrt{d}$ changes the numerator by at most a constant times $\vol_{d-1}(\partial A_m)$. 
	
	For $(i)\Rightarrow(ii)$, suppose that there is no BD-map between $\Lambda_0$ and $\Lambda_1$, that is, no bijection $\phi:\Lambda_0\to\Lambda_1$ that satisfies 
	\[\sup_{\xx\in\Lambda_0}\norm{\xx-\phi(\xx)}_\infty<\infty.\]

	For every $m\in\N$ consider the bipartite graph $\GG_m\df(\Lambda_0 \sqcup \Lambda_1, \EE_m)$, where 
	\[ \EE_m = \Big\{ \{\xx,\yy\} \mid \xx\in\Lambda_0, \yy\in\Lambda_1, \norm{\xx-\yy}_\infty\le 2m \Big\}.\]
	The existence of a perfect matching in $\GG_m$ for some $m$ would imply the existence of a BD-map between $\Lambda_0$ and $\Lambda_1$, contradicting our assumption. Thus by Hall's marriage theorem (see e.g. \cite{Rado}), for every $m\in\N$ there is a set $X_m\subset \Lambda_{i_m}$, $i_m\in\{0,1\}$, so that $\# X_m > \#(X_m^{(+2m)} \cap\Lambda_{1-i_m})$. Fix $m\in\N$, and assume without loss of generality that $i_m=0$.
	Set 
	\begin{equation*}\label{eq:A_m def}
	A_m\df\bigcup\{Q\in\QQ_d(m) \mid Q\cap X_m \neq\varnothing\}\in\QQ_d^*(m).
	\end{equation*}
	For $Q\in\QQ_d(m)$ let $Q'$ be a cube of edge-length $3m$ which is concentric with $Q$, and set 
	\[B_m\df\bigcup\{Q' \mid Q\in\QQ_d(m), Q\cap X_m \neq\varnothing\}\in\QQ_d^*(m).\]
	Clearly $B_m\supset A_m\supset X_m$, and by the triangle inequality we have $X_m^{(+2m)}\supset B_m$. Therefore 
	\[\#(\Lambda_0\cap A_m)>\#(\Lambda_1 \cap B_m)=\#(\Lambda_1 \cap A_m)+\#(\Lambda_1 \cap (B_m\smallsetminus A_m)),\]
	which implies 
	\begin{equation*}\label{eq:our_Lemma2.3-1}
		\#(\Lambda_0 \cap A_m) - \#(\Lambda_1 \cap A_m)> \#(\Lambda_1 \cap (B_m\smallsetminus A_m)).	
	\end{equation*}
	
	It is left to show that $\#(\Lambda_1 \cap (B_m\smallsetminus A_m))/\vol_{d-1}(\partial A_m)\xrightarrow{m\to\infty}\infty$, which is a consequence of the following argument, taken from the proof of \cite[Lemma 2.3]{Laczk}. Suppose that $\partial A_m$ consists of $s$ faces of cubes in $\QQ_d(m)$. For each such face, let $P_j$ be the cube in $\QQ_d(m)$ contained in $B_m\smallsetminus A_m$ with boundary containing that face. Note that $P_1,\ldots,P_s$ are not necessarily distinct and that each cube has $2d$ faces, and so 
	\begin{equation*}\label{eq:Laczkovich_argument_from_Lemma2.3}
	2d\cdot \vol(B_m\smallsetminus A_m)\ge \sum_{j=1}^s\vol(P_j) = s\cdot m^d = m\cdot s\cdot m^{d-1} = m\cdot \vol_{d-1}(\partial A_m).
	\end{equation*}
	The relative denseness of $\Lambda_1$ implies that $\#(\Lambda_1 \cap (B_m\smallsetminus A_m)) \ge c\cdot\vol(B_m\smallsetminus A_m)$ for some constant $c>0$ independent of $m$, and the proof follows. 
\end{proof}

\begin{cor}\label{cor:A_m_contain_large_balls}
	Let $(A_m)_{m\in\N}$ be a sequence of sets as in  \eqref{eq:non_BD_consequence}, then for every $R>0$ there exists $M>0$ so that for every $m\ge M$ each $A_m$ contains a ball of radius $R$.   
\end{cor}

\begin{proof}
	Let $R>0$ and suppose that there is a sequence $m_j\to\infty$ such that for every $j$ the set $A_{m_j}$ does not contain a ball of radius $R$. Then for every $j$ we have $A_{m_j} \subset (\partial A_{m_j})^{(+R)}$ and thus by Lemma \ref{lem:Lacz.Lemmas_2.1+2.2}
	\begin{equation*}\label{eq:large_discrepancy==>large_balls1}
		\vol(A_{m_j}) \le c_0\cdot R^d\cdot \vol_{d-1}(\partial A_{m_j}).
	\end{equation*}
	Since $\Lambda_0$ and $\Lambda_1$ are uniformly discrete and relatively dense, there exist constants $a,b>0$ so that for every $j$
	\begin{equation*}\label{eq:large_discrepancy==>large_balls2}
		a\cdot \vol(A_{m_j}) \le \#(\Lambda_0 \cap A_{m_j}), \#(\Lambda_1 \cap A_{m_j}) \le b\cdot \vol(A_{m_j}).
	\end{equation*}
	Combining the above implies that for every $j$ we have
	\[
		\frac{| \#(\Lambda_0 \cap A_{m_j}) - \#(\Lambda_1 \cap A_{m_j}) |}{\vol_{d-1}(\partial A_{m_j})} \le (b-a)c_0\cdot R^d,
	\] 
	contradicting \eqref{eq:non_BD_consequence}.
\end{proof}

\section{The topology on spaces of Delone sets}\label{sec:topology}
We consider the dynamical system $(X,d,G)$, where $(X,d)$ is a compact metric space and $G$ is a group acting on $X$. The dynamical system $(X,d,G)$ is called \emph{minimal} if every \emph{$G$-orbit}, $G.x \df \{g.x \mid g\in G\}$ for $x\in X$, is dense in $(X,d)$. A set $S \subset G$ is called \emph{syndetic} if there is a compact set $K \subset G$ so that for every $g\in G$ there is a $k\in K$ with $kg\in S$. Note that when $G=\R^d$ this notion coincides with our definition of a relatively dense set.  A point $x_0\in X$ is said to be \emph{uniformly recurrent} if for every open neighborhood $U$ of $x_0$ the set of `return times' to $U$, $\{g\in G \mid g.x_0\in U \}$, is syndetic. As shown in \cite[Theorem 1.15]{Furstenberg}, in minimal systems every point is uniformly recurrent. 

Recall that given a metric $\rho$ on $\R^d$ we may use \eqref{eq:CF-metric} to define a metric $D$ on $\CCC(\R^d)$, the space of closed subsets of $(\R^d,\rho)$, and that this metric induces the  \textit{Chabauty--Fell topology}. Here and in what follows we take $\rho$ to be the metric defined by the supremum norm $\norm{\cdot}_\infty$  on $\R^d$. Note that replacing it with any other norm on $\R^d$, such as the Euclidean norm, would change the metric $D$ but not the induced topology,  also known as the \textit{local rubber topology} in the context of aperiodic order. It is known that $D$ is a complete metric on $\CCC(\R^d)$, and the space $\left(\CCC(\R^d),D\right)$ is compact, see e.g. \cite{de la Harpe}, \cite{Lenz-Stollmann}.

Let $\X$ be a collection of Delone sets in $\R^d$. Under the additional assumptions that $\X$ is a closed subset of $\CCC(\R^d)$ and that $\R^d$ acts on $\X$ by translations, the space $(\X,D,\R^d)$ is a compact dynamical system. We say that $\Lambda\in\X$ is \emph{almost repetitive} if for every $\xx\in\R^d$ and $\varepsilon>0$ there exists $R = R(\varepsilon,\xx)>0$ such that every ball $B(\yy,R)$ in $\R^d$ contains a vector $\vv\in\R^d$ that satisfies
\[D(\Lambda - \xx, \Lambda- \vv)<\varepsilon.\]
In words, for every $\xx\in\R^d$ and $\varepsilon>0$ there exists $R>0$ so that a copy of $B({\bf0},1/\varepsilon)\cap (\Lambda-\xx)$ can be found in every $R$-ball, up to wiggling each point by at most $\varepsilon$. 
We also refer to \cite[Definitions 2.8, 2.13, 3.5]{Frettloh-Richard} and to \cite{Lagarias-Pleasants} for distinctions between similar definitions of repetitivity. 

The observation in Lemma \ref{lem:discrepancy_on_subsets} is useful when working with the metric $D$ in spaces of uniformly discrete point sets. 
\begin{lem}\label{lem:discrepancy_on_subsets}
	Suppose that $\Lambda_0, \Lambda_1\subset \R^d$ are uniformly discrete sets with separation constant $r>0$, and that $D\left(\Lambda_0, \Lambda_1\right) < \varepsilon$ for $0<\varepsilon<r/2$. Then for every set $A \subset B({\bf0},1/\varepsilon)$ that is a translated copy of an element of $\QQ_d^*$, there exist injective maps 
	\[\varphi_0:\Lambda_0 \cap A \to \Lambda_1 \cap A^{(+\varepsilon)}, \quad  \varphi_1:\Lambda_1 \cap A \to \Lambda_0 \cap A^{(+\varepsilon)}, \]
	that satisfy
	\begin{equation}\label{discrepancy_on_subsets1}
	\forall \xx\in \Lambda_0 \cap A:\quad  \norm{\xx - \varphi_0(\xx)}_\infty<\varepsilon, 
	\qquad 
	\forall \yy\in \Lambda_1 \cap A:\quad  \norm{\yy - \varphi_1(\yy)}_\infty<\varepsilon.
	\end{equation}
	In particular, there is a constant $c_1$ that depends on $d$ and $r$ so that 
	\begin{equation}\label{discrepancy_on_subsets2}
	\absolute{\# (\Lambda_0 \cap A) - \#(\Lambda_1 \cap A) } \le c_1\cdot \varepsilon^d \cdot \vol_{d-1}(\partial A).
	\end{equation}
\end{lem}

\begin{proof}
	Given $A\subset B({\bf0},1/\varepsilon)$ as above, since $D\left(\Lambda_0, \Lambda_1\right) < \varepsilon$, the existence of $\varphi_0, \varphi_1$ satisfying \eqref{discrepancy_on_subsets1} follows directly from the definition of $D$ in \eqref{eq:CF-metric}. Note that the maps are injective since $\varepsilon<r/2$. Therefore 
	\[
		\absolute{\# (\Lambda_0 \cap A) - \#(\Lambda_1 \cap A) } \le 
		\#\left( \Lambda_0 \cap (\partial A)^{(+\varepsilon)} \right) +
		\#\left( \Lambda_1 \cap (\partial A)^{(+\varepsilon)} \right).
	\]
	Since $\Lambda_0$ and $\Lambda_1$ are uniformly discrete and in view of Lemma \ref{lem:Lacz.Lemmas_2.1+2.2}, \eqref{discrepancy_on_subsets2} follows.
\end{proof}

We remark that if $\Lambda$ is a Delone set in $\R^d$ with separation constant and packing radius $r,R>0$, and if $\X$ is the orbit closure of $\Lambda$ with respect to $D$, then every $\Gamma\in \X$ is a Delone set with separation constant at least $r$ and packing radius at most $R$. 

The following lemma shows that minimal spaces are \emph{uniformly almost repetitive}. Namely, the radius $R(\xx,\varepsilon)$ from the definition of almost repetitivity above does not depend on $\xx$.  
 
\begin{lem}\label{lem:uniform_repetitivity}
	Let $\X$ be a compact space of Delone sets so that the dynamical system $(\X,D,\R^d)$ is minimal. Then for every $0<\varepsilon<1$ there exists $R=R(\varepsilon)>0$, so that for every $\Lambda, \Gamma\in\X$  and $\yy\in\R^d$, there exists some $\vv\in B(\yy,R)$ for which
	\[D(\Gamma, \Lambda - \vv)<\varepsilon.\] 
\end{lem}

\begin{proof}
	Let $\varepsilon>0$, and let $\Lambda\in\X$ and $\xx\in\R^d$. By minimality, the set $\Lambda -\xx$ is uniformly recurrent. For $\eta>0$ denote
	$U^\xx_{\eta} \df \{\Lambda' \in\X \mid D(\Lambda-\xx, \Lambda')<\eta \}$, then the set $\{ \vv \in \R^d \mid \Lambda - \vv \in U^\xx_{\varepsilon/2} \}$ is relatively dense (syndetic). In other words, there exists $R^\xx_{\varepsilon/2}> 0$ such that every cube of edge-length $R^\xx_{\varepsilon/2}$ in $\R^d$ contains some $\vv\in\R^d$ satisfying 
	$D(\Lambda - \xx, \Lambda - \vv) < \varepsilon/2$. 
	
	By minimality again, the collection $\{\Lambda - \xx \mid \xx\in\R^d \}$ is dense in $\X$. Thus $\{U^\xx_{\varepsilon/2}\}_{x\in\R^d}$ is an open cover of $\X$, and by compactness there exists a finite sub-cover $U^{\xx_1}_{\varepsilon/2},\ldots, U^{\xx_n}_{\varepsilon/2}$.  Then for every $\Gamma\in\X$ these exists some $j\in\{1,\ldots,n\}$ so that $\Gamma \in U^{\xx_j}_{\varepsilon/2}$, and hence $D(\Gamma,\Lambda-\xx_j)<\varepsilon/2$. Setting
	$R\df \max\{R^{\xx_1}_{\varepsilon/2}, \ldots, R^{\xx_n}_{\varepsilon/2} \}$, it follows that for every $\yy\in\R^d$ there exists some $\vv\in B(\yy,R^{\xx_j}_{\varepsilon/2}) \subset B(\yy,R)$ such that $D(\Lambda - \xx_j,\Lambda - \vv)<\varepsilon/2$. Then by the triangle inequality $D(\Gamma, \Lambda - \vv)<\varepsilon$, as required.	 
\end{proof}

 In Proposition \ref{thm:constructing_Delone_sets_as_limits_in_CF} below we consider a Delone set $\Lambda$ in a minimal space, and show that if sets $A_m$ in $\QQ_d^*$ grow sufficiently fast, then there exist translation vectors $\uu_m$ so that the patches $Q_m = (\Lambda\cap A_m)-\uu_m$ converge to a limit object  that  ``almost'' contains all of the $Q_m$'s. The idea of the proof is simply to use the almost repetitivity property to inductively find  an ``almost'' copy of $Q_{m-1}$ inside $\Lambda\cap A_m$, and to set $\uu_m$ so that it is centered accordingly, namely so that the copy we find ``almost'' agrees with $Q_{m-1}$.  
 Note that every sequence of sets that grows in a reasonable sense has a subsequence that grows fast enough to satisfy conditions (1) and (2) in Proposition \ref{thm:constructing_Delone_sets_as_limits_in_CF}.

\begin{prop}\label{thm:constructing_Delone_sets_as_limits_in_CF}
	Let $\X$ be a minimal space of Delone sets in $\R^d$, $\Lambda\in\X$, $(A_m)_{m\in\N}$ a sequence of sets in $\QQ_d^*$ and $(\varepsilon_m)_{m\ge 0}$  a decreasing sequence of positive constants with $\varepsilon_0<\min\{1,r(\Lambda)/2\}$, where $r(\Lambda)$ is the separation constant of $\Lambda$. For every $m\ge 0$ choose $R(\varepsilon_m)$ satisfying Lemma \ref{lem:uniform_repetitivity} and set $R_m\df \max\{R(\varepsilon_m), 1/\varepsilon_m)\}$. Assume that the following properties hold for every $m\in\N$:
	\begin{enumerate}
		\item
		There exists $\xx_m\in\R^d$ such that $A_m\subset B(\xx_m,1/2\varepsilon_{m})$. 
		\item
		There exists $\yy_{m}\in\R^d$ such that $B(\yy_{m},2R_{m-1}) \subset A_{m}$. 
	\end{enumerate}
	 Then there exist $\uu_m \in B(\yy_{m}, R_{m-1})$ and patches $Q_m\df (\Lambda \cap A_m) -\uu_m$ such that $\lim_{m\to\infty}Q_m = \Gamma\in\overline{\R^d.\Lambda}=\X$. Moreover, for every $m\ge 2$ 
	 
	 \begin{equation}\label{eq:patch_convergences-goal0}
		 B({\bf0},R_{m-1}) \subset A_m - \uu_m \subset B({\bf0},1/\varepsilon_{m}),
	 \end{equation}
	 
	 \begin{equation}\label{eq:patch_convergences-goal1}
	 D(\Lambda-\uu_{m-1},\Lambda-\uu_{m})<\varepsilon_{m-1},
	 \end{equation}
	 
	 \begin{equation}\label{eq:patch_convergences-goal2}
		 D(Q_m,\Gamma) < \varepsilon_{m-1}
	 \end{equation} 
	 and there exists $c_2>0$ so that
	 \begin{equation}\label{eq:patch_convergences-goal3}
		 \absolute{\# \left(\Gamma\cap(A_{m}-\uu_{m})\right) - \# Q_m} \le 
		 c_2 \cdot \varepsilon_{m}^d \cdot \vol_{d-1}(\partial A_m),
	 \end{equation}
	 	where $c_2$ depends on the dimension $d$ and separation constant $r(\Lambda)$.
\end{prop}

\begin{proof}
	First observe that by assumptions (1) and (2) 
	\begin{equation}\label{eq:patch_convergences-eps_decay_geometrically}
		\varepsilon_{m+1}\le \frac{1}{4R_{m}} \le \frac{1}{4}\varepsilon_m
	\end{equation}
	holds for every $m\in\N$. 
	In particular, the series $\sum_{m=1}^{\infty}\varepsilon_m$ is convergent.
	
	We define the vectors $\uu_m$, and hence the patches $Q_m$, inductively. 
	\begin{itemize}
		\item 
		By (1), $A_1$ is in particular contained in a ball of radius $1/\varepsilon_1$. Let $\uu_1$ be such that $Q_1= (\Lambda \cap A_1) -\uu_1$ is contained in $B({\bf0},1/\varepsilon_1)$. 	
	\end{itemize}
	Assume that the vectors $\uu_j$, and thus the patches $Q_j = (\Lambda \cap A_j) - \uu_j$, are defined for $j\in\{1, \ldots, m\}$ such that for every $2\le j\le m$ we have 
	\begin{enumerate}
		\item[(i)]
		$B({\bf0},R_{j-1}) \subset A_j - \uu_j \subset B({\bf0},1/\varepsilon_j)$.
		\item[(ii)]
		$D(\Lambda-\uu_j,\Lambda-\uu_{j-1})<\varepsilon_{j-1}$.
	\end{enumerate}
	We define $\uu_{m+1}$ as follows.
	\begin{itemize}
		\item 
		By (2), $A_{m+1}$ contains a ball of the form $B(\yy_{m+1},2R_{m})$. By Lemma \ref{lem:uniform_repetitivity}, let $\uu_{m+1}\in B(\yy_{m+1},R_{m})$ be a vector satisfying  
		\begin{equation*}\label{eq:patch_convergences-def_of_u_m+1}
			D(\Lambda-\uu_m, \Lambda-\uu_{m+1})<\varepsilon_m.
		\end{equation*} 
		Thus (ii) for $j = m+1$ holds. 
		Note that since $B(\yy_{m+1},2R_{m}) \subset A_{m+1}$ and $\uu_{m+1}\in B(\yy_{m+1},R_{m})$ we have 
		\begin{equation*}\label{eq:patch_convergences-A_m_contains_large_ball}
			B({\bf0},R_{m}) \subset A_{m+1} - \uu_{m+1}.
		\end{equation*} 
		By (1), $A_{m+1} - \uu_{m+1} \subset B(\xx_{m+1}-\uu_{m+1},1/2\varepsilon_{m+1})$ and so $A_{m+1} - \uu_{m+1}$ contains the origin. Then by the triangle inequality, $A_{m+1} - \uu_{m+1}$ is contained in $B({\bf0},1/\varepsilon_{m+1})$, completing the proof of (i) for $j=m+1$.	
	\end{itemize}
	This completes the construction of the vectors $\uu_m$ and the patches $Q_m$. Next we show that the sequence $(Q_m)_{m\in\N}$ is a Cauchy sequence. Fix some $\varepsilon>0$ and let $M$ be so that $2\varepsilon_M<\varepsilon$. Let $m>n>M$, and note that by property (ii) we have $D(\Lambda-\uu_{k+1},\Lambda-\uu_k)<\varepsilon_k$, for every $k\ge M$. Then by the triangle inequality, 
	\begin{equation}\label{eq:patch_convergences-Cauchy_seq1}
	D(\Lambda-\uu_m,\Lambda-\uu_n)\le \sum_{k=n}^{m-1}D(\Lambda-\uu_{k+1},\Lambda-\uu_{k})< \sum_{k=n}^{m-1}\varepsilon_k < 2\varepsilon_n< 2\varepsilon_M<\varepsilon,
	\end{equation}
	where the third inequality follows from \eqref{eq:patch_convergences-eps_decay_geometrically}. By property (i), for every $j\in\N$ the point sets $Q_j$ and $\Lambda - \uu_j$ in particular coincide on the ball $B({\bf0},1/\varepsilon_{j-1})$. Since $m,n>M$, the sets $\Lambda - \uu_n$ and $Q_n$ coincide on $B({\bf0},1/\varepsilon)$, and similarly for $\Lambda-\uu_m$ and $Q_m$. Therefore, relying on \eqref{eq:patch_convergences-Cauchy_seq1}, for every $m>n>M$ we have 
	\begin{equation}\label{eq:patch_convergences-Cauchy_seq2}
	D(Q_m,Q_n)\le D\Big((\Lambda-\uu_m)\cap B({\bf0},1/\varepsilon),(\Lambda-\uu_n)\cap B({\bf0},1/\varepsilon)\Big) < \varepsilon. 
	\end{equation}
	Thus $(Q_m)_{m\in\N}$ is a Cauchy sequence. The space $(\X,D)$ is complete, as a compact metric space, hence the limit $\Gamma \df \lim_{m\to\infty}Q_m = \lim_{m\to\infty}\Lambda-\uu_m$ exists and belongs to $\X$. 
	
	It is left to prove \eqref{eq:patch_convergences-goal0}, \eqref{eq:patch_convergences-goal1}, \eqref{eq:patch_convergences-goal2} and \eqref{eq:patch_convergences-goal3}. 
	First observe that \eqref{eq:patch_convergences-goal0} and \eqref{eq:patch_convergences-goal1} follow immediately from the construction, see properties (i) and (ii). 
	To see \eqref{eq:patch_convergences-goal2}, let $m\in\N$ and let $k>m$ be so that $D(Q_k,\Gamma)<\varepsilon_m$. Repeating the computations in \eqref{eq:patch_convergences-Cauchy_seq1} and \eqref{eq:patch_convergences-Cauchy_seq2} yields that $D(Q_m,Q_k)<2\varepsilon_m$, and by \eqref{eq:patch_convergences-eps_decay_geometrically} we have 
	\[
	D(Q_m,\Gamma)\le D(Q_m,Q_k)+D(Q_k,\Gamma) < 3\varepsilon_m < \varepsilon_{m-1}.
	\]
	
	Finally, we prove \eqref{eq:patch_convergences-goal3}. By \eqref{eq:patch_convergences-goal0} we have $A_m - \uu_m \subset B({\bf0},1/\varepsilon_m)$ and by \eqref{eq:patch_convergences-goal2} we have  $D(\Gamma,Q_{m+1}) < \varepsilon_m$. Thus by Lemma \ref{lem:discrepancy_on_subsets} with $A = A_m - \uu_m$  we obtain 
	\begin{equation}\label{eq:patch_convergences-discrepancy1}
		 \absolute{\#(\Gamma \cap (A_m - \uu_m)) - \#(Q_{m+1} \cap (A_m - \uu_m))} \le 
	 c_1\cdot\varepsilon_m^d\cdot \vol_{d-1}(\partial A_m).
	\end{equation}
	By \eqref{eq:patch_convergences-goal1} we have $D(\Lambda-\uu_{m},\Lambda-\uu_{m+1})<\varepsilon_{m}$, and applying Lemma \ref{lem:discrepancy_on_subsets} once again we get 
	\[
		\absolute{\#\left((\Lambda-\uu_{m+1}) \cap (A_m - \uu_m)\right) - \#\left( (\Lambda-\uu_{m}) \cap (A_m - \uu_m)\right) } \le 
		c_1\cdot\varepsilon_m^d\cdot \vol_{d-1}(\partial A_m).
	\]
	By the definition of the $Q_m$'s, and since $A_m - \uu_m \subset A_{m+1}-\uu_{m+1}$ by	\eqref{eq:patch_convergences-goal0}, this is exactly 
	\begin{equation*}\label{eq:patch_convergences-discrepancy2}
		\absolute{\#\left(Q_{m+1} \cap (A_m - \uu_m)\right) - \#Q_{m}  } \le 
		c_1\cdot\varepsilon_m^d\cdot \vol_{d-1}(\partial A_m).
	\end{equation*}
	Combining this with \eqref{eq:patch_convergences-discrepancy1}  yields \eqref{eq:patch_convergences-goal3} and completes the proof of the theorem.
\end{proof}

\section{Finding patches with large discrepancy} \label{sec:construction_of_the_patches}
The goal of this section is to prove the following proposition, which will be used in our proof of Theorem \ref{thm:main_result_general} in \S \ref{sec:proof of main}.

\begin{prop}\label{prop:finding_traget_patches}
	Let $\Lambda\subset\R^d$ be a non-uniformly spread Delone set. Then there exist a sequence $(A_m)_{m\in\N}$ of sets in $\QQ_d^*$ and a sequence $(\xx_m)_{m\in\N}$ of vectors in $\Z^d$ so that 
	 \begin{equation}\label{eq:finding_traget_patches}
		 \frac{\absolute{ \#(\Lambda \cap A_m) - \#(\Lambda \cap (A_m+\xx_m)) }}{\vol_{d-1}(\partial A_m)} \xrightarrow{m\to\infty} \infty.	 
	 \end{equation}
\end{prop}
Let $\Lambda \subset \R^d$ be a Delone set. We define the \emph{central lower density} and the \emph{central upper density} of $\Lambda$ respectively by 
\[\Delta_*(\Lambda) \df \liminf_{t\to\infty} \frac{\#\left(B({\bf0},t) \cap\Lambda\right)}{\vol(B({\bf0},t))}
\qquad 
\Delta^*(\Lambda) \df \limsup_{t\to\infty} \frac{\#\left(B({\bf0},t)\cap\Lambda\right)}{\vol(B({\bf0},t))}.\]
If the limit  $\lim_{t\to\infty} \#\left(B({\bf0},t)\cap\Lambda\right)/\vol(B({\bf0},t))$ exists, it is called the \emph{central density} of $\Lambda$ and is denoted by $\Delta(\Lambda)$.

We begin with the following lemma.

\begin{lem}\label{lem:averaging_densities}
	Let $\Lambda$ be a Delone set, $\gamma>0$ and $A\in\QQ_d^*$. Then for every $\varepsilon>0$ there exists $K>0$ such that for every integer $k\ge K$:
	\begin{itemize}
	\item[(1)]
	if $\frac{\#(\Lambda \cap B({\bf0},k))}{\vol(B({\bf0},k))} \ge \gamma$ then the ball $B({\bf0},k)$ contains  $A+\xx$, a translated copy of $A$ with $\xx\in\Z^d$, such that 
	\[\frac{\#(\Lambda \cap (A+\xx))}{\vol(A)}\ge \gamma-\varepsilon.\] 
	\item[(2)]
	if $\frac{\#(\Lambda \cap B({\bf0},k))}{\vol(B({\bf0},k))} \le \gamma$ then the ball $B({\bf0},k)$ contains  $A+\xx$, a translated copy of $A$ with $\xx\in\Z^d$, such that 
	\[\frac{\#(\Lambda \cap (A+\xx))}{\vol(A)}\le \gamma+\varepsilon.\] 
	\end{itemize}
\end{lem}

\begin{proof}
	This is a simple averaging argument. We prove property (1), the proof of property (2) is similar.
	
		Denote by $\rho$ the diameter of the set $A$. 
	For a large integer $k$ we write 
	$B({\bf0},k) = B_{[-\rho]} \sqcup (\partial B)_{[+\rho]}$, where $B_{[-\rho]}, (\partial B)_{[+\rho]}\in\QQ_d^*$ are defined by 
	\begin{equation}\label{eq:averaging_lemma-decomposition_of_B(0,k)}
	\begin{aligned}
	B_{[-\rho]}&\df 
	\bigcup\left\{Q\in\QQ_d \mid Q\subset B({\bf0},k), \dist(Q,\partial B({\bf0},k))>\rho\right\}\\(\partial B)_{[+\rho]}&\df B({\bf0},k)\smallsetminus B_{[-\rho]},
		\end{aligned}
	\end{equation}
	where $\dist(X,Y) \df \inf\{\norm{\xx - \yy}_\infty \mid \xx\in X,  \yy\in Y\}$. 
	
	Given $\varepsilon>0$ we pick $K\in\N$ large enough so that for every integer $k\ge K$ we have
	\begin{equation}\label{eq:averaging_lemma-0}
	\frac{\vol\left((\partial B)_{[+\rho]}\right)}{\vol(B({\bf0},k))}<\frac{\varepsilon}{2}.
	\end{equation}
	
	Let $k\ge K$ such that  
\begin{equation}\label{eq:averaging_lemma-1}
	\frac{\#(\Lambda \cap B({\bf0},k))}{\vol(B({\bf0},k))}\ge \gamma,
\end{equation}
	and let $N_k \df\{\xx\in\Z^d \mid A+\xx\subset B({\bf0},k) \}$. 
	By way of contradiction, assume that 
	\begin{equation}\label{eq:averaging_lemma-2}
	\forall \xx\in N_k:\quad  \#\left(\Lambda \cap (A+\xx)\right)< (\gamma-\varepsilon)\vol(A). 
	\end{equation}
	
	Notice that the number of cubes from $\QQ_d$ that form $A$ is $\vol(A)$. Then by counting the points of $\Lambda$ (with multiplicity) in all the sets $A+\xx$, $\xx\in N_k$, the points in every unit lattice cube in $B_{[-\rho]}$ is counted exactly $\vol(A)$ times. Thus 
	\begin{equation}\label{eq:averaging_lemma-3}
	\#N_k(\gamma - \varepsilon)\vol(A)
	 \stackrel{\eqref{eq:averaging_lemma-2}}> 
	\sum_{\xx\in N_k} \#(\Lambda \cap (A+\xx)) \ge 
	\vol(A)\cdot  \# \left( \Lambda \cap B_{[-\rho]} \right).
	\end{equation}
	Note that $\# N_k\le\vol(B({\bf0},k))$, then dividing both sides of \eqref{eq:averaging_lemma-3} by $\vol(A)\cdot\vol(B({\bf0},k))$ yields 
\[
\gamma - \varepsilon>
\frac{\# \left(\Lambda \cap B_{[-\rho]} \right)}{\vol(B({\bf0},k))} \stackrel{\eqref{eq:averaging_lemma-decomposition_of_B(0,k)}}\ge
\frac{\# \left( \Lambda \cap B({\bf0},k) \right)}{\vol(B({\bf 0},k))} - 
\frac{\# \left( \Lambda \cap (\partial B)_{[+\rho]} \right)}{\vol(B({\bf0},k))} \stackrel{\eqref{eq:averaging_lemma-0},  \eqref{eq:averaging_lemma-1}}>
\gamma - \frac{\varepsilon}{2}, 
\]
a contradiction.
\end{proof}

\begin{lem}\label{lem:no_density==>cubes_of_different_density}
	Suppose that $\Lambda$ is a Delone set in $\R^d$ and that $\Delta_*(\Lambda)<\Delta^*(\Lambda)$. Then there exist $\alpha<\beta$, integers $a_k\to\infty$ and $\xx_k\in\Z^d$ such that 
	\[
	\frac{\#\left( \Lambda \cap B({\bf0},a_k) \right)}{\vol(B({\bf0},a_k))} \le \alpha \quad\text{ and }\quad 
	\frac{\#\left( \Lambda \cap B(\xx_k,a_k) \right)}{\vol(B(\xx_k,a_k))} \ge \beta.
	\] 
\end{lem}

\begin{proof}
	By the assumption on the densities, there exist sequences $a_k, b_l\to\infty$  so that 
	\[
	\lim_{k\to\infty}\frac{\#\left( \Lambda \cap B({\bf0},a_k) \right)}{\vol(B({\bf0},a_k))} = \tilde\alpha 
	\quad\text{ and }\quad 
	\lim_{l\to\infty}\frac{\#\left( \Lambda \cap B({\bf0},b_l) \right)}{\vol(B({\bf0},b_l))} = \tilde\beta,
	\]
	where $\tilde\alpha<\tilde\beta$. Since $\Lambda$ is uniformly discrete, and since the $(d-1)$-volume of the boundary of a cube grows slower than the cube's volume, we may assume that the numbers $a_k,b_k$ are integers. Let $\delta<\frac{\tilde\beta-\tilde\alpha}{3}$ and fix $K\in\N$ such that for every $k,l\ge K$ we have 
	\begin{equation}\label{eq:no_density==>cubes_different_density-1}
	\frac{\#\left( \Lambda \cap B({\bf0},a_k) \right)}{\vol(B({\bf0},a_k))} \le \tilde\alpha +\delta 
	\quad\text{ and }\quad 
	\frac{\#\left( \Lambda \cap B({\bf0},b_l) \right)}{\vol(B({\bf0},b_l))} \ge \tilde\beta - \delta.
	\end{equation}
	For every $k$, applying Lemma \ref{lem:averaging_densities} with $A=B({\bf0},a_k)$, $\varepsilon=\frac{\tilde\beta-\tilde\alpha}{3}-\delta>0$, and $\tilde\beta-\delta$ in the role of $\gamma$, and combining this with \eqref{eq:no_density==>cubes_different_density-1}, we find a large enough $l=l_k$ and $\xx_k\in\Z^d$ so that $B({\bf0},b_l)$ contains the ball $B(\xx_k,a_k)$, which satisfies    
	\begin{equation}\label{eq:no_density==>cubes_different_density-2}
		\frac{\#\left( \Lambda \cap B(\xx_k,a_k) \right)}{\vol(B(\xx_k,a_k))} \ge (\tilde\beta - \delta) - \varepsilon = \tilde\beta -\frac{\tilde\beta-\tilde\alpha}{3}.
	\end{equation}
	Setting $\alpha\df \tilde\alpha + \frac{\tilde\beta-\tilde\alpha}{3}$ and $\beta\df \tilde\beta - \frac{\tilde\beta-\tilde\alpha}{3}$, the assertion follows from \eqref{eq:no_density==>cubes_different_density-1} and \eqref{eq:no_density==>cubes_different_density-2}.
\end{proof}

\begin{proof}[Proof of Proposition \ref{prop:finding_traget_patches}]
	Let $\Lambda\subset\R^d$ be a non-uniformly spread Delone set. 
	In view of Lemma \ref{lem:no_density==>cubes_of_different_density} we may further assume that $\Delta \df \Delta(\Lambda)$ exists. 
	For $\alpha\neq \Delta^{-1/d}$ the Delone sets $\alpha\Z^d$ and $\Lambda$ do not have the same central density and hence there is no BD-map between them (see e.g. \cite[Corollary 3.2]{FSS}). By our assumption on $\Lambda$, there is no BD-map between $\Lambda$ and $\Delta^{-1/d}\Z^d$ as well. Applying Theorem \ref{thm:non-BD_consequence} on these two Delone sets we obtain a sequence  $(A_m)_{m\in\N}$ of sets in $ \QQ_d^*$ that satisfies 
	%\begin{equation}
	\[
		\frac{\absolute{ \#(\Delta^{-1/d}\Z^d \cap A_m) - \#(\Lambda \cap A_m) } }{\vol_{d-1}(\partial A_m)} \xrightarrow{m\to\infty} \infty.		
	\]
	%\end{equation} 
	By passing to a subsequence of $(A_m)_{m\in\N}$ we may assume that 
	\begin{equation}\label{eq:finding_patches-1}
	\frac{ \#(\Delta^{-1/d}\Z^d \cap A_m) - \#(\Lambda \cap A_m) }{\vol_{d-1}(\partial A_m)} \xrightarrow{m\to\infty} \infty, 		
	\end{equation} 
	and complete the proof using (1) of Lemma \ref{lem:averaging_densities}. In the case that $\#(\Delta^{-1/d}\Z^d \cap A_m) < \#(\Lambda \cap A_m)$  for all large values of $m$, the proof is similar using (2) of Lemma \ref{lem:averaging_densities} instead of (1).
	
	For every $m\in\N$ we pick $\varepsilon_m$ such that 
	\begin{equation}\label{eq:finding_patches-2}
		\varepsilon_m \vol(A_m)< \vol_{d-1}(\partial A_m)	
	\end{equation}
	and apply Lemma \ref{lem:averaging_densities} with $\gamma = \Delta - \varepsilon_m$, $A=A_m$ and $\varepsilon=\varepsilon_m$. Note that since $\Delta(\Lambda) = \Delta$ exists, the condition $\frac{\#(\Lambda \cap B({\bf0},k))}{\vol(B({\bf0},k))} \ge \Delta - \varepsilon_m$ is satisfied for any sufficiently large $k$. By (1) of Lemma \ref{lem:averaging_densities}, in particular, there exists 
	a vector $\xx_m\in\Z^d$ so that   
	\begin{equation}\label{eq:finding_patches-3}
	\frac{\#\left( \Lambda \cap (A_m+\xx_m) \right)}{\vol(A_m)}\ge \Delta - 2\varepsilon_m.
	\end{equation}
	By \eqref{eq:finding_patches-1}
	\begin{equation}\label{eq:finding_patches-4}
	\frac{ \#(\Delta^{-1/d}\Z^d \cap A_m) - \#\left( \Lambda \cap (A_m+\xx_m) \right) }{\vol_{d-1}(\partial A_m)} + 
	\frac{ \#\left( \Lambda \cap (A_m+\xx_m) \right) - \#(\Lambda \cap A_m) }{\vol_{d-1}(\partial A_m)}
	 \xrightarrow{m\to\infty} \infty.
	\end{equation} 
	Note that 
	\[
	\#(\Delta^{-1/d}\Z^d \cap A_m) \le \Delta \cdot \vol(A_m) + %(\sqrt{d} \cdot \Delta^{-1/d})
	c\cdot \vol_{d-1}(\partial A_m),
	\]
	where $c$ depends on $d$ and $\Delta$, and by \eqref{eq:finding_patches-3} we also have 
	\[
	\left( \Lambda \cap (A_m+\xx_m) \right) \ge (\Delta - 2\varepsilon_m)\vol(A_m).
	\]
	Then 
	\begin{equation*}\label{eq:finding_patches-5}
	\begin{aligned}
		\#(\Delta^{-1/d}\Z^d \cap A_m) - \#\left( \Lambda \cap (A_m+\xx_m) \right) & \le c \cdot \vol_{d-1}(\partial A_m) + 2\varepsilon_m\vol(A_m) \\ 
		& \stackrel{\eqref{eq:finding_patches-2}}\le c'\cdot  \vol_{d-1}(\partial A_m),	
	\end{aligned}
	\end{equation*}
	where $c'$ depends on $d$ and $\Delta$.  Plugging this in \eqref{eq:finding_patches-4} completes the proof.
\end{proof}

\section{Proof of Theorem \ref{thm:main_result_general}}\label{sec:proof of main}

Given a non-uniformly spread Delone set $\Lambda \subset \R^d$, let $A_m\in \QQ_d^*$ and $\xx_m \in \Z^d$ be as in Proposition \ref{prop:finding_traget_patches}.  Let $\varepsilon_m >0$ be so that $A_m$ is contained in a ball of radius $1/2\varepsilon_m$. Passing to subsequences, by Corollary \ref{cor:A_m_contain_large_balls} combined with \eqref{eq:finding_traget_patches} we may assume that $A_{m}$ contains a ball of radius $2R_{m-1}$, where $R_m$ is as in Proposition \ref{thm:constructing_Delone_sets_as_limits_in_CF}. We thus have 
\begin{equation}\label{eq:A_m_and_B_m_1}
B\left( \yy_m,2R_{m-1}\right) \subset A_m \subset B\left( \zz_m, 1/2\varepsilon_m \right)
\end{equation} 
for some $\yy_m, \zz_m\in\R^d$. Denote  
\begin{equation}\label{eq:B_m,p_m,q_m_def}
B_m \df A_m + \xx_m, \qquad \pp_m \df \yy_m +\xx_m, \qquad \qq_m \df \zz_m + \xx_m.
\end{equation}
Then 
\begin{equation}\label{eq:A_m_and_B_m_2}
B\left( \pp_m,2R_{m-1}\right) \subset B_m \subset B\left( \qq_m, 1/2\varepsilon_m \right)
\end{equation} 
and so $(A_m)_{m\in\N}$ and $(B_m)_{m\in\N}$ both satisfy Proposition \ref{thm:constructing_Delone_sets_as_limits_in_CF}. 

By \eqref{eq:finding_traget_patches}, there is a sequence of constants $\mu_m\to\infty$ such that 
\begin{equation}\label{eq:mu_m_def}
\absolute{ \#(\Lambda \cap A_m) - \#(\Lambda \cap (A_m+\xx_m)) } = \mu_m \cdot \vol_{d-1}(\partial A_m).
\end{equation} 
Since $\mu_m\to\infty$, by passing to a further subsequence, 
%mutually for $A_m, \xx_m, \varepsilon_m$ and $\mu_m$, 
we may assume that $\mu_m$ approaches infinity at an extremely fast rate. In particular, by defining every element in the sequence with dependence on the previous one, we may assume that 
\begin{equation}\label{eq:mu_m_propery}
	\frac{R_{m-1}^d}{\mu_m}\xrightarrow{m\to\infty} 0.
\end{equation} 

Using these notations, Theorem \ref{thm:main_result_general} follows from Lemmas \ref{lem:constructing_Lambda_omega} and \ref{lem:distinct_seq.==>non-BD_Delone_sets} below.

\begin{lem}\label{lem:constructing_Lambda_omega}
	Let $\X$ be a minimal space of Delone sets and assume that there exists $\Lambda\in\X$ that is non-uniformly spread. Let $(A_m)_{m\in\N}$ and $(B_m)_{m\in\N}$ be the sequences of sets in $\QQ^*_d$ defined in Proposition \ref{prop:finding_traget_patches} and in \eqref{eq:B_m,p_m,q_m_def}, with respect to $\Lambda$. For every word $\omega \in \{A,B\}^\N$ let $(C_m)_{m\in\N}$ be the sequence of sets in $\QQ_d^*$ defined by
	\begin{equation}\label{eq:C_m_def}
		C_m\df 
		\begin{cases}
		A_m,& \omega(m) = A \\
		B_m,& \omega(m) = B	,
		\end{cases} 	
	\end{equation}
	where $w(m)$ is the $m$'th letter in $w$. Then there exists a sequence $(\uu_m)_{m\in\N}$ of vectors in $\R^d$ so that $\Lambda_\omega = \lim_{m\to\infty} (\Lambda \cap C_m) - \uu_m$ is a Delone set in $\X$, 
	\begin{equation}\label{eq:u_m_def-constructing_Lambda_omega}
		\uu_m \in 
		\begin{cases}
		B(\yy_m, R_{m-1}), & \omega(m) = A \\
		B(\pp_m, R_{m-1}),& \omega(m) = B,
		\end{cases}
	\end{equation}
	and 
	\begin{equation}\label{eq:discrepancy_in_Lambda_omega}
	 \forall m\ge 2: \quad\absolute{\#(\Lambda_\omega \cap (C_m - \uu_m))-\#(\Lambda \cap C_m) } \le c_3\cdot \vol_{d-1}(\partial C_m),
	\end{equation}
	where $c_3$ is a constant that depends on $d$ and on the separation constant $r(\Lambda)$.
\end{lem}

\begin{proof}
	Given $\omega \in \{A,B\}^\N$, consider the sequence $(C_m)_{m\in\N}$ of sets in $\QQ_d^*$ defined by \eqref{eq:C_m_def}. By \eqref{eq:A_m_and_B_m_1} and \eqref{eq:A_m_and_B_m_2}, conditions (1) and (2) of Proposition \ref{thm:constructing_Delone_sets_as_limits_in_CF} are being satisfied for $(C_m)_{m\in\N}$, with $(\varepsilon_m)_{m\in\N}$ as described at the beginning of this section. Applying Proposition \ref{thm:constructing_Delone_sets_as_limits_in_CF} we obtain vectors $\uu_m$ satisfying \eqref{eq:u_m_def-constructing_Lambda_omega},  for which the sequence of patches $Q_m \df (\Lambda \cap C_m) - \uu_m$ is convergent. Setting $\Lambda_\omega$ to be the limit set, by 
	\eqref{eq:patch_convergences-goal3} of Proposition \ref{thm:constructing_Delone_sets_as_limits_in_CF} for every $m\ge 2$ 
	\begin{equation*}\label{eq:small_Hausdorff_dist==>small_discrepancy}
		\absolute{\#\left(\Lambda_\omega \cap (C_m - \uu_m)\right) - \# Q_m} \le c_3 \cdot \varepsilon_m^d \cdot \vol_{d-1}(\partial C_m), 
	\end{equation*}
	where $c_3$ depends on $d$ and on $r(\Lambda)$. 
	Clearly $\# Q_m = \#(\Lambda \cap C_m)$, and \eqref{eq:discrepancy_in_Lambda_omega} follows.
\end{proof}

\begin{lem}\label{lem:distinct_seq.==>non-BD_Delone_sets}
	Let $\X$ be a minimal space of Delone sets and assume that there exists $\Lambda\in\X$ that is non-uniformly spread.  Let $\eta, \sigma\in \{A,B\}^\N$ be two words that differ in infinitely many places. Then the Delone sets $\Lambda_\eta$ and $\Lambda_\sigma$ defined in Lemma \ref{lem:constructing_Lambda_omega} are BD-non-equivalent.	
\end{lem}

\begin{proof}
	Taking a subsequence if necessary, we may assume without lose of generality that $\eta$ and $\sigma$ are everywhere different.
	We use an upper index of $\eta$ or $\sigma$ on elements of $\QQ_d^*$ and on vectors, e.g. $C_m^\eta$ and $\uu_m^\sigma$, to distinguish between those elements that come from the construction of $\Lambda_\eta$ and of $\Lambda_\sigma$ in Lemma \ref{lem:constructing_Lambda_omega}.  
	
	Denote $F_m \df C_m^\eta - \uu_m^\eta$. By \eqref{eq:discrepancy_in_Lambda_omega} for $w=\eta$ we obtain
	\begin{equation}\label{eq:estimating_F_m_cap_Lambda_eta}
		\forall m\ge 2: \quad 
		\absolute{\#(\Lambda_\eta \cap F_m) - \#(\Lambda \cap C_m^\eta)} \le c_3\cdot\vol_{d-1}({\partial F_m}).
	\end{equation}
	Observe that for every $m\ge 2$  there exists some $\vv_m\in\R^d$ so that 
	\begin{equation}\label{eq:v_m_def}
		[F_m - \vv_m = ]\quad(C_m^\eta - \uu_m^\eta) - \vv_m = C_m^\sigma - \uu_m^\sigma \quad\text{ and }\quad \norm{\vv_m}_\infty \le 2R_{m-1}.	
	\end{equation}
	 Indeed, assume without loss of generality that $\eta(m) = A$ and $\sigma(m) = B$. Combining \eqref{eq:B_m,p_m,q_m_def},  \eqref{eq:C_m_def} and \eqref{eq:u_m_def-constructing_Lambda_omega} yields that $C_m^\eta = A_m, C_m^\sigma = A_m + \xx_m$, $\uu_m^\eta \in B(\yy_m, R_{m-1})$ and $\uu_m^\sigma \in B(\yy_m + \xx_m, R_{m-1})$, which implies \eqref{eq:v_m_def}. It follows that
	\[
		\forall m\ge 2: \quad 
		(C_m^\sigma - \uu_m^\sigma)\triangle F_m \subset \partial F_m^{(+2 R_{m-1})},
	\] 
	and hence by \eqref{eq:Lacz.Lemmas_2.1+2.2}
	\[
		\forall m\ge 2: \quad 
		\absolute{\#(\Lambda_\sigma \cap F_m) - \#(\Lambda_\sigma \cap (C_m^\sigma - \uu_m^\sigma))} \le c_4 \cdot R_{m-1}^d\cdot \vol_{d-1}({\partial F_m}),
	\]
	where $c_4$ depends on $d$ and on $r(\Lambda)$. Again by \eqref{eq:discrepancy_in_Lambda_omega}, this time with $w=\sigma$, we obtain 
	\begin{equation*}\label{eq:estimating_F_m_cap_Lambda_sigma}
			\forall m\ge 2: \quad 
			\absolute{\#(\Lambda_\sigma \cap F_m) - \#(\Lambda \cap C_m^\sigma)} \le \left(c_3+c_4 \cdot R_{m-1}^d \right)\vol_{d-1}({\partial F_m}).
	\end{equation*}
	Combining this with \eqref{eq:estimating_F_m_cap_Lambda_eta}, the triangle inequality yields that for every $m\ge 2$ 
	\begin{equation*}\label{eq:final_discrepancy_estimate1}
	\begin{aligned}
		&\absolute{\#(\Lambda_\eta \cap F_m) - \#(\Lambda_\sigma \cap F_m)} \ge 
		\\
		&\absolute{\#(\Lambda \cap C_m^\eta) - \#(\Lambda \cap C_m^\sigma)} - \absolute{\#(\Lambda_\eta \cap F_m) - \#(\Lambda \cap C_m^\eta)}  -
		\absolute{\#(\Lambda_\sigma \cap F_m) - \#(\Lambda \cap C_m^\sigma)} \ge
		\\
		&\absolute{\#(\Lambda \cap C_m^\eta) - \#(\Lambda \cap C_m^\sigma)} - c_5\cdot R_{m-1}^d \cdot \vol_{d-1}(\partial F_m),
	\end{aligned}
	\end{equation*}
	where $c_5$ depends on $d$ and $r(\Lambda)$. 
	Since $C_m^\eta = A_m$, $C_m^\sigma = A_m + \xx_m$ and $\vol_{d-1}(\partial A_m) = \vol_{d-1}(\partial F_m)$, combined with \eqref{eq:mu_m_def} we have 
	\begin{equation*}\label{eq:final_discrepancy_estimate3}
	\begin{aligned}
		\absolute{\#(\Lambda_\eta \cap F_m) - \#(\Lambda_\sigma \cap F_m)} \ge \left(\mu_m - c_5\cdot R_{m-1}^d\right) \vol_{d-1}(\partial F_m),
	\end{aligned}
	\end{equation*}
	and together with \eqref{eq:mu_m_propery} we thus obtain 
	\[
		\frac{\absolute{\#(\Lambda_\eta \cap F_m) - \#(\Lambda_\sigma \cap F_m)}}{\vol_{d-1}(\partial F_m)} \ge 
		\mu_m\left(1 - \frac{c_5\cdot R_{m-1}^d}{\mu_m}\right)
		\xrightarrow{m\to\infty} \infty.
	\]
	Theorem \ref{thm:non-BD_consequence} then implies that the  sets $\Lambda_\eta$ and $\Lambda_\sigma$ are BD-non-equivalent, as required. 
\end{proof}

\begin{proof}[Proof of Theorem \ref{thm:main_result_general}]
	Let $\X$ be a minimal space of Delone sets. If there exists a uniformly spread $\Lambda\in\X$, then as noted in \S \ref{sec:introduction} every $\Lambda \in\X$ is uniformly spread, and (1) holds. 
	
	Otherwise, there exists some $\Lambda \in\X$ that is non-uniformly spread. Consider the equivalence relation on $\{A,B\}^\N$ in which $\eta \sim \sigma$ if $\eta$ and $\sigma$ differ in only finitely many places, and let $\Omega\subset \{A,B\}^\N$ be a set of equivalence class representatives. Since every equivalence class in this relation is countable, $\absolute{\Omega} = \cont$. For every two distinct words $\eta, \sigma \in \Omega$, Lemma \ref{lem:distinct_seq.==>non-BD_Delone_sets} implies that $\Lambda_\eta$ and $\Lambda_\sigma$ are BD-non-equivalent, therefore $\bdX \ge \cont$. As explained in \S \ref{sec:introduction} the upper bound is trivial, and so the proof is complete.  
\end{proof}

\appendix \section{$D(\cdot,\cdot)$ is a metric}\label{Appendix}
It is known that the function $D(\cdot,\cdot)$ in \eqref{eq:CF-metric} constitutes a metric on $\CCC(\R^d)$ when it is capped by $1/\sqrt{2}$ instead of $1$, see e.g. \cite[\S 7]{Lee-Solomyak}. We show that it is indeed a metric also when capped by $1$. The proof is essentially the same. 

\begin{prop}
	The function $D(\cdot,\cdot)$ in \eqref{eq:CF-metric} is a metric on $\CCC(\R^d)$. 
\end{prop}
\begin{proof}
	The triangle inequality is the only property that is not immediate. Let $X, Y, Z\in\CCC(\R^d)$ be three closed sets and let $\varepsilon, \delta>0$ so that 
	\begin{equation}\label{eq:triangle_ineq-assumption}
	D(X,Y)\le \varepsilon \text{ and } D(Y,Z)\le \delta.
	\end{equation}
	We must show that $D(X,Z)\le \varepsilon+\delta$. Clearly this is true if $\varepsilon+\delta\ge1$, and so we may assume in what follows that $\varepsilon+\delta<1$. We have 
	\begin{equation}\label{eq:triangle_ineq-1}
	X\cap B\left({\bf 0}, \frac{1}{\varepsilon+\delta}\right) = 
	X\cap B\left({\bf 0}, \frac{1}{\varepsilon}\right) \cap B\left({\bf 0}, \frac{1}{\varepsilon+\delta}\right) \stackrel{\eqref{eq:triangle_ineq-assumption}}\subset 
	Y^{(+\varepsilon)} \cap B\left({\bf 0}, \frac{1}{\varepsilon+\delta}\right)
	\end{equation}
	Note that since $\delta< \varepsilon+\delta<1$, the expression
	\[
	\frac{1}{\delta} - \frac{1}{\varepsilon+\delta} - \varepsilon = \frac{\varepsilon+\delta - \delta - \varepsilon\delta(\varepsilon+\delta)}{\delta(\varepsilon+\delta)} = \
	\frac{\varepsilon(1-\delta(\varepsilon+\delta))}{\delta(\varepsilon+\delta)}
	\]
is positive, and so
%\begin{equation}\label{eq:triangle_ineq-2}
$\varepsilon + \frac{1}{\varepsilon+\delta} < \frac{1}{\delta}$.
%\end{equation}
By the triangle inequality% for the metric $\rho$
, if $B(\xx,\varepsilon)\cap B\left({\bf 0},\frac{1}{\varepsilon+\delta}\right) \neq \varnothing$ then 
$\rho({\bf 0},\xx)<\varepsilon + \frac{1}{\varepsilon+\delta} <\frac{1}{\delta}$.
Therefore 
%\begin{equation}\label{eq:triangle_ineq-3}
\[
Y^{(+\varepsilon)} \cap B\left({\bf 0}, \frac{1}{\varepsilon+\delta}\right)
\subset \left(Y\cap B\left({\bf 0},\frac{1}{\delta}\right)\right)^{(+\varepsilon)}\stackrel{\eqref{eq:triangle_ineq-assumption}}\subset \left(Z^{(+\delta)}\right)^{(+\varepsilon)} \subset Z^{(+(\varepsilon+\delta))},
\]
Combining this with \eqref{eq:triangle_ineq-1} finishes the proof. 
\end{proof}

\end{document}